\documentclass[12pt]{amsart}
\usepackage{mathrsfs}
\usepackage{latexsym,amssymb,amsmath,amsfonts,graphicx,epstopdf,subfigure,float}
\usepackage{bbm}
\usepackage{pifont}
\usepackage{cite}
\usepackage{multirow}
\usepackage{appendix}
\usepackage{tikz,xcolor}

\usetikzlibrary{decorations.markings}
\usetikzlibrary{intersections}
\usetikzlibrary{calc}

\def\dist{\mathop{\rm dist}\nolimits}

\newtheorem{theorem}{Theorem}[section]
\newtheorem{thm}[theorem]{Theorem}

\newtheorem{cor}{Corollary}[section]
\newtheorem{lemma}{Lemma}[section]
\newtheorem{remark}{Remark}[section]
\newtheorem{defi}{Definition}[section]

\newtheorem{example}{Example}[section]
\numberwithin{equation}{section}

\def\R{\mathbb R}

\def\mathscr{\mathcal }

\def\diag{\text{diag}}

\newcommand{\bomega}{{\boldsymbol{\omega}}}

\newcommand{\bd}{{\mathbf{d}}}

\newcommand{\bi}{{\mathbf{i}}}
\newcommand{\bj}{{\mathbf{j}}}

\newcommand{\bp}{{\mathbf{p}}}

\newcommand{\bu}{{\mathbf{u}}}
\newcommand{\bv}{{\mathbf{v}}}
\newcommand{\bw}{{\mathbf{w}}}
\newcommand{\bx}{{\mathbf{x}}}
\newcommand{\by}{{\mathbf{y}}}

\newcommand{\SD}{{\mathcal D}}
\newcommand{\SE}{{\mathcal E}}

 \usepackage{color}

\begin{document}
 \title[Point-wise doubling indices of   measures]{Point-wise doubling indices of   measures and its application to
 bi-Lipschitz classification of Bedford-McMullen carpets}

 \author{Hui Rao} \address{Department of Mathematics and Statistics, Central China Normal University, Wuhan, 430079, China
} \email{hrao@mail.ccnu.edu.cn
 }

 \author{Yan-Li Xu$^*$} \address{Department of Mathematics and Statistics, Central China Normal University, Wuhan, 430079, China
} \email{xu\_yl@mails.ccnu.edu.cn
 }

\author{Yuan Zhang}
\address{Department of Mathematics and Statistics, Central China Normal University, Wuhan, 430079, China}
\email{yzhang@mail.ccnu.edu.cn}

\date{\today}
\thanks {The work is supported by NSFC No. 11971195. }


\thanks{* Corresponding author.}

\begin{abstract}
Doubling measure was introduced by Beurling and Ahlfors
 in 1956 and now it becomes a basic concept in analysis on metric space.
In this paper, for a measure which is not doubling, we introduce a notion of point-wise doubling index, and calculate the point-wise doubling indices of uniform Bernoulli measures on
Bedford-McMullen carpets. As an application, we show that, except a small class of Bedford-McMullen carpets,  if two Bedford-McMullen carpets
are bi-Lipschitz equivalent, then they have the same fiber sequence up to a permutation.
\end{abstract}
\maketitle

\section{\textbf{Introduction}}

A Borel probability measure $\mu$ on a metric space
is said to be \emph{doubling} if
$$\underset{x\in \text{supp}(\mu),~r>0}{\sup}~\dfrac{\mu(B_{2r}(x))}{\mu(B_r(x))}
<\infty,$$
where $B_r(x)$ denotes the open ball with center $x$ and radius $r$, and $\text{supp}(\mu)$, the support of $\mu$, consists of $x$ such that $\mu(B_r(x))>0$
for every $r>0$.

Doubling measure is introduced by  Beurling and Ahlfors \cite{Beurling_1956} in 1956.
Since then,  it is an important object in the analysis of metric spaces, see for instance  \cite{Beurling_1956, Stein_1993, Koskela_2000, Heinonen_2001} \textit{etc}.
Usually, a measure is considered to be `nice' if it is doubling. But if a measure is not doubling, what can we say?
The purpose of this paper is to study changing rate of a non-doubling measure point-wisely.
Precisely, we introduce the following notion.

  \begin{defi}\textbf{(Point-wise doubling index)}
\emph{Let  $\mu$ be a measure on a metric space $(X,d)$.
  Let $0<\rho<1$. Let $\varphi:(0,1)\to \R^+$ be a     strictly decreasing  function satisfying $\lim_{r\to 0}\varphi(r)=\infty$,
  which we call a \emph{gauge function}.  For $z\in X$, the \emph{upper doubling index} of $\mu$ at $z$
  (w.r.t. $\varphi$) is defined by
 \begin{equation}\label{def:delta}
 \overline{\delta}_{\varphi}(z; \mu,\rho)=\limsup_{r\rightarrow0}\sup_{\substack {B_{\rho r}(z')\subset B_{r}(z)}}\frac{\log[\mu(B_{r}(z))/\mu(B_{\rho r}(z'))]}{\varphi(r)};
 \end{equation}
 the \emph{lower doubling index} $\underline{\delta}_\varphi(z; \mu,\rho)$ is obtained by replacing $\limsup$ by $\liminf$ in \eqref{def:delta}.
 If the two values agree, the  common value is called the \emph{doubling index} and  denoted by $\delta_\varphi(z; \mu,\rho)$. }
\end{defi}

\begin{remark}\label{doubling_index}
\emph{ It is seen that if $\mu$ is a doubling measure, then $\delta_\varphi(z; \mu,\rho) \equiv 0$ for all $z\in X$.}
\end{remark}

In this paper, we calculate doubling indices of uniform Bernoulli measures
on Bedford-McMullen carpets, and  we show that  they are  Lipschitz invariants. Thanks to these new invariants, we obtain some remarkable new results concerning the bi-Lipschitz classification of Bedford-McMullen carpets (see Theorem \ref{thm_5} and Theorem \ref{thm:fiber}).
Furthermore, we show that two Lipschitz equivalent Bedford-McMullen carpets
have the same fiber sequence up to a permutation except a small class (see Theorem \ref{thm:best}).

Let $n>m\geq 2$ be two integers and denote by $\diag (n,m)$ the $2\times2$ diagonal matrix. Let ${\SD}\subset\{0,1,\dots,n-1\}\times\{0,1,\dots,m-1\}$.
For $\bd\in\SD$, set
\begin{equation}\label{affineMap}
S_\bd(z)=\diag(n^{-1}, m^{-1})(z+\bd),
\end{equation}
then $\{S_\bd\}_{\bd\in \SD}$ is an iterated function system (IFS), its invariant set $E=K(n,m,{\SD})$, the unique non-empty compact set  satisfying $E=\bigcup_{\bd\in \SD} S_\bd(E)$, is called a \emph{Bedford-McMullen carpet (abbreviated as BM-carpet)}\cite{Bed84,M84}.
We denote by $\mu_E$ the unique Borel probability measure  supported on $E$  satisfying
$$\mu_E(\cdot)=\dfrac{1}{\# \SD}\underset{\bd\in \SD}{\sum} \mu_E\circ S_\bd^{-1}(\cdot),$$
 and  call $\mu_E$  the \emph{uniform Bernoulli measure} of $E$,
 where $\#A$ denotes the cardinality of a set $A$.

Recall that two metric spaces $(X,d_{X})$ and $(Y,d_{Y})$ are said to be \emph{Lipschitz equivalent}, denoted by $(X,d_{X})\sim(Y,d_{Y})$, if there exists a map $f:X\rightarrow Y$ which is bi-Lipschitz, that is, there is a constant $C > 0$ such that
$$
C^{-1}d_{X}(x,y)\leq d_{Y}(f(x),f(y))\leq Cd_{X}(x,y), \quad \text{for all}\ x,y\in X.
$$

 Recently, there are many works devoted to the Lipschitz classification of BM-carpets,
 see \cite{Miao2013,HZ2024,Miao2017, Rao2019, YangZh21, B, YangZh23}.  It turns out that the classification   is much  more complicated than the self-similar settings.
  Notably,  there are  a lot of Lipschitz invariants,
which divide the BM-carpets  into  many sub-families invariant under bi-Lipschitz maps.
  The most well-known Lipschitz  invariants are  various dimensions, such as  Hausdorff dimension, box dimension, Assouad dimension, \textit{etc}.
  Fraser and Yu \cite{FH18} showed that $\log m/\log n$ is a Lipschitz invariant, Rao, Yang and Zhang \cite{Rao2019},
 and Banaji and  Kolossv\'ary \cite{B}
showed that the multifractal spectrum  of the uniform Bernoulli measure on BM-carpet is a Lipschitz invariant, see also  Huang, Rao, Wen and Xu\cite{HRWX22}.

Let $E=K(n,m,\SD)$ be a BM-carpet.
We define
$$
a_{j}=\#\{i:(i,j)\in\SD\}, \quad  j=0,1,\dots, m-1,
$$
and call $(a_j)_{j=0}^{m-1}$ the \emph{fiber sequence} of $\SD$. Throughout the paper, we denote
$$
\sigma={\log m}/{\log n}. 
$$

 We say a BM-carpet is of \emph{non-doubling type} if the associated uniform Bernoulli measure  is not doubling.
Li, Wei and Wen \cite{LWW16} characterized when a Bernoulli measure on a BM-carpet is doubling.
According to their result, we have

\begin{lemma} The uniform Bernoulli measure $\mu_E$ is not doubling if and only if
 $a_0a_{m-1}>0$,  $a_0\neq a_{m-1}$, and
 $a_ja_{j+1}> 0$ for at least one $j\in \{0,1,\dots,m-2\}$.
\end{lemma}

For $\bi=\bd_{1}\ldots \bd_{k}\in\SD^k$, we define $S_{\bi}=S_{\bd_{1}}\circ\cdots\circ S_{\bd_{k}}$.
Let $z\in E$, we call $(\bd_j)_{j\geq 1}$ a \emph{coding} of $z$ if
 $\{z\}=\bigcap_{k\ge 1} S_{\bd_1\dots \bd_k}([0,1]^2).$
For convenience, we denote $\bd_j=(x_j, y_j)$ and we also write a coding as $(\bx,\by)=((x_j)_{j\geq 1}, (y_j)_{j\geq 1})$.

\begin{defi}[\cite{Miao2013} \cite{YangZh23}] \emph{We call $z$ a \emph{double vertical coding point} if $z$ has two codings $(\bx,\by)$ and $(\bx',\by')$ such that $\by\neq\by'$.  Denote
 \begin{equation}
 V_E=\{z\in E;~ z \text{ has double vertical codings}\}.
 \end{equation}
 A BM-carpet  $E$  is said to satisfy the \emph{vertical separation condition (VSC)} if $V_E=\emptyset$.}
\end{defi}

Li, Li and Miao \cite{Miao2013} proved that if $E=K(n,m,\SD)$ and $F=K(n,m, \SD')$ are totally disconnected, satisfy the VSC and share the same fiber sequence, then $E\sim F$.  Yang and Zhang \cite{YangZh23} strengthen this result by showing that  if a BM-carpet is totally disconnected and satisfies the VSC, then it is Lipschitz equivalent to a certain symbolic space.

Now we state our main results. For the lower doubling index, we have

\begin{thm}\label{thm:lower}
Let $E=K(n, m, \SD)$ be a BM-carpet of non-doubling type and assume that $a_0>a_{m-1}$.
Let $0<\rho<n^{-3}$ and  let $\varphi$ be a gauge function satisfying
\begin{equation}\label{eq:rank0}
\lim_{k\to \infty} \frac{k}{\varphi(n^{-k})}=s\in [0,+\infty].
\end{equation}
Then
$$
\underline{\delta}_\varphi(z; \mu_E,\rho)=\left \{
\begin{array}{ll}
s(1/\sigma-1)\log(a_0/a_{m-1}), &\text{  if $z\in V_E$}, \\
  0, &\text{ otherwise.}
  \end{array}
  \right .
  $$
\end{thm}

\begin{remark} \emph{ If the gauge function $\varphi$ is too big such that the limit in \eqref{eq:rank0} is zero, then $\underline{\delta}_\varphi(z; \mu_E,\rho)$ is constantly $0$, and it is not interesting.}
\end{remark}

 For $j\in \{0,1,\dots,m-1\}$, set
$
\Omega_j=\{(\bx,\by)\in \SD^\infty;$$~\text{ $\by$ ends with $j^\infty$}\}$.
Clearly
$$V_E\subset \Omega_0\cup \Omega_{m-1}.$$

To compute the upper doubling index, we  introduce a function $\beta(k;\bomega)$, a kind of   reverse run length function for a  $\bomega\in \SD^\infty$ at position $k$,  in \eqref{eq:beta} of Section 3.

\begin{thm}\label{thm:formula}
Let $E=K(n, m, \SD)$ be a BM-carpet of non-doubling type and assume that $a_0>a_{m-1}$.
Let $\varphi$ be a gauge function satisfying
\begin{equation}\label{eq:rank}
\lim_{k\to \infty} \frac{k}{\varphi(n^{-k})}=s\in (0,+\infty].
\end{equation}
Let $0<\rho< n^{-3}$. Let $z\in E$ and let $\bomega$ be a coding of $z$.
Then

(i) If $\bomega\not\in \Omega_0\cup \Omega_{m-1}$, then
$$
\overline{\delta}_\varphi(z; \mu_E,\rho)=\limsup_{k\to \infty} \frac{\beta(k;\bomega)}{\varphi(n^{-k})}\cdot \log (a_0/a_{m-1}) .
$$

(ii) If $\bomega\in \Omega_0\cup \Omega_{m-1}$, then
$
{\overline\delta}_\varphi(z; \mu_E,\rho)=s(1/\sigma-1)\log(a_0/a_{m-1})\cdot {\mathbf 1}_{V_E}(z)
$,  where ${\mathbf 1}_{V_E}$ is the indicator function of $V_E$.
\end{thm}

\begin{remark}\emph{From the above two theorems, we see that under the assumption \eqref{eq:rank}, $\overline{\delta}_\varphi$ and $\underline{\delta}_\varphi$
are irrelevant with $\rho$. So from now on, we will use $\overline{\delta}_\varphi(z; \mu_E)$ and
$\underline{\delta}_\varphi(z; \mu_E)$ instead of $\overline{\delta}_\varphi(z; \mu_E,\rho)$ and
$\underline{\delta}_\varphi(z; \mu_E,\rho)$ respectively}.
\end{remark}


\begin{example}\label{exam-1} \emph{ Let $E=K(n, m, \SD)$ be a BM-carpet of non-doubling type and assume that $a_0>a_{m-1}$. We are especially interested in the gauge functions $\varphi(r)=-\log r$ and $\varphi(r)=\log|\log r|$.}

\emph{
(1)  Let $\varphi(r)=-\log r$, and denote $\delta=\delta_\varphi$. Then
\begin{equation}\label{formula_delta}
\overline{\delta}(z; \mu_E)=\left \{
\begin{array}{ll}
\underset{k\rightarrow \infty}{\limsup}\dfrac{\beta(k;\bomega)}{ k}\cdot \log_n ({a_0}/{a_{m-1}}), & \text{if} ~~\bomega\notin \Omega_0\cup \Omega_{m-1},\\
 \delta_{\max} \cdot {\mathbf 1}_{V_E}(z), & \text{ otherwise },
\end{array}
\right .
\end{equation}
where $\delta_{\max}=(\log n)^{-1}(1/\sigma-1)\log ( a_0/ a_{m-1})$, and
 $\underline{\delta}(z;\mu_E)= \delta_{\max} \cdot {\mathbf 1}_{V_E}(z).$}

\emph{Moreover, we can show that for $\mu_E$-a.e. $z\in E$, ${\delta}(z; \mu_E)=0$,
and the range of $\overline{\delta}(z; \mu_E)$ is $[0,  \delta_{\max}]$. (See Theorem \ref{thm:delta}.) }

\emph{
(2) Let $\varphi(r)=\log|\log r|$, and denote $\Delta=\delta_\varphi$. Then
\begin{equation}\label{formula_Delta}
\overline{\Delta}(z; \mu_E)=\left \{
\begin{array}{ll}
\underset{k\rightarrow \infty}{\limsup}\dfrac{\beta(k;\bomega)}{\log k}\cdot\log ({a_0}/{a_{m-1}}), & \text{if} ~~\bomega\notin \Omega_0\cup \Omega_{m-1},\\
\infty \cdot {\mathbf 1}_{V_E}(z), & \text{ otherwise },
\end{array}
\right .
\end{equation}
and
 $\underline{\Delta}(z;\mu_E)=\infty \cdot {\mathbf 1}_{V_E}(z).$ }

\emph{ Moreover, we have that for $\mu_E$-a.e. $z\in E$,
\begin{equation}\label{eq_full_Mea1}
  \overline{\Delta}(z; \mu_E)=- \log_{p_0}(a_0/a_{m-1}):=\Delta_{aver},
\end{equation}
where $p_0=a_0/\#\SD$,  which tells us that $\varphi(r)=\log|\log r|$ is the  gauge function for almost every  $z\in E$. (See Theorem \ref{thm:Delta}.)
}
\end{example}


In the following, we  use the point-wise doubling indices to construct new Lipschitz invariants
for classification of  BM-carpets.  We assume that $E=K(n,m, \SD), F=K(n,m,\SD')$ are two BM-carpets.
We use $(a_j)_{j=0}^{m-1}$ and $(a'_j)_{j=0}^{m-1}$ to denote the fiber sequences
of $\SD$ and $\SD'$, respectively.
Denote by $\mu_E$ and $\mu_F$ the uniform Bernoulli measures of $E$ and $F$ respectively.

The following result shows that the VSC is a Lipschitz invariant.

\begin{thm}\label{thm_5}
Let $E=K(n, m, \SD),F=K(n, m, \SD')$ be two BM-carpets  of non-doubling type.
Let $f:E\rightarrow F$ be a bi-Lipschitz map.   Then for any $z\in E$,
\begin{equation}\label{eq_MainThm3}
\overline \delta_\varphi(z; \mu_E)= \overline \delta_\varphi(f(z); \mu_F),
\quad \underline \delta_\varphi(z; \mu_E)= \underline \delta_\varphi(f(z); \mu_F).
\end{equation}
Consequently,
$$V_F=f(V_E) \text{ and  }  \dim_H V_E=\dim_H V_F.$$ Especially,   $E$ satisfies the VSC (that is, $V_E=\emptyset$) if and only if $F$  does.
\end{thm}

Using the Lipschitz invariants $\delta_{\max}$, $\Delta_{aver}$ and $\gamma_{\max}$ (which is defined
in Section 6), we show that

\begin{thm}\label{thm:fiber}
Let $E=K(n, m, \SD),F=K(n, m, \SD')$ be two BM-carpets of non-doubling type.
Suppose  $E\sim F$ . Then

(i) the  fiber sequence of $E$ is a permutation of
 that of $F$;

 (ii) we have $a_0=a_0'~\text{and} ~a_{m-1}=a_{m-1}'$
 if we assume without loss of generality that $a_0>a_{m-1}$ and  $a_0'>a_{m-1}'$ .
\end{thm}

We use ${\mathcal M}_{t,v,d,r}(n,m)$ to denote the class of BM-carpets
which are totally disconnected (which is indicated by $t$) with expanding matrix $\text{diag}(n,m)$,
possess vacant rows (indicated by $v$),
are of doubling type (indicated by $d$), and $\log m/\log n$ is rational (indicated by $r$).
It is well-known that these four properties are all Lipschitz invariants.

\begin{thm}\label{thm:best}
Let $E=K(n, m, \SD),F=K(n, m, \SD')$ be two BM-carpets
which are not in ${\mathcal M}_{t,v,d,r}(n,m)$.
Then  $E\sim F$ implies that  the  fiber sequence of $E$ is a permutation of
 that of $F$.
\end{thm}

\begin{remark} \emph{Let $E=K(n,m,\SD)$ and $F=K(n,m,\SD')$ be two BM-carpets. We remark that if they share the same fiber sequence up to a permutation, then $\mu_E$ and $\mu_F$ share the same multifractal spectrum, and this further implies
 that they share the same Hausdorff, box, Assouad and intermediate  dimensions (\cite{B, Rao2019}).}
 \end{remark}

\begin{remark}\emph{Set
 $H=\{i; (i,j),(i,j+1)\in \SD\}$  and  $I=\{i; ~~(i,0),(i,m-1)\in\SD\}.$
It is easy to show that if $H\neq \emptyset$ and $\#{I}\geq 2$, then
$
\dim_H V_E={\log (\#I)}/{\log n};
$
otherwise, $V_E$ is  either countable or empty.
}
\end{remark}

\begin{example}
\emph{Let $E=K(8,4,\SD)$ and $F=K(8,4,\SD')$ be two BM-carpets with digit sets indicated in Fig.\ref{E}.
Notice that the fiber sequence of $\SD$ is a permutation of that of $\SD'$.
They are indeed not Lipschitz equivalent, which can be obtained either by $\dim_H V_E=0\neq \dim_H V_F=\log 2/\log 8$, or by $a_0\neq a_0'$.
}
\begin{figure}[hppt]
  \begin{tikzpicture}[xscale=.5,yscale=1.0]
        \pgfmathsetmacro{\h}{12}
        \foreach \i in {\h,0} \draw(\i,0)grid++(8,4);
        \foreach \i in {{0,0},{1,0},{7,0},{3,1},{4,1},{6,1},{7,1},{2,2},{4,2},{5,2},{6,2},{1,3},{2,3}}
        \draw[fill=blue!45](\i)rectangle++(1,1);
        \foreach \i in {{0,0},{1,0},{2,0},{7,0},{4,1},{5,1},{6,1},{2,2},{3,2},{4,2},{5,2},{0,3},{7,3}}
        \draw[fill=blue!45]($(\i)+(\h,0)$) rectangle++(1,1);
  \end{tikzpicture}
  \text{\emph{(a)~$E$: $a_0=3$, $a_{m-1}=2$. ~~~~~~~~~~~\quad\qquad(b)~$F$: $a'_0=4$, $a'_{m-1}=2$.}}
 \caption{}
 \label{E}

\end{figure}
\end{example}
%

\begin{remark}
\emph{Recently, there are a lot of works devoted to the Assouad dimension
of measures on metric spaces, see \cite{KJ2013, Frazer2020, FH2020}.
R. Anttila \cite{Anttila2023} introduced a notion of pointwise Assouad dimension of a measure, which provides  another pointwise index of measures on metric spaces.
}
\end{remark}

The paper is organized as follows.
We recall some known results about approximate squares of BM-carpets in Section \ref{sec2}.
In Section \ref{mainresult1}, we prove several important lemmas.
Theorem \ref{thm:lower} and  Theorem \ref{thm:formula} are proved in Section \ref{proved Main1-2}.
We give some remarks on $\delta_{-\log r}$   and $\delta_{\log|\log r|}$ in Section \ref{mainR2_Coro1}.
We discuss an alternative point-wise doubling index in Section \ref{mainresult2}.
Theorems \ref{thm_5}, \ref{thm:fiber} and \ref{thm:best} are proved in Section \ref{mainresult3}.

\section{\textbf{Estimates of measures of approximate squares }}\label{sec2}
Let $E = K(n, m, \SD)$ be a BM-carpet. Throughout the paper, we denote
$N=\#\SD$ and use the notation $\ell(k)=\lfloor k/\sigma\rfloor$, where $\lfloor x\rfloor$ denotes the greatest integer no larger than $x$.
Clearly,
$$
m^{\ell(k)}\leq n^k<m^{\ell(k)+1}.
$$
For $\bi\in\SD^k$, we call
 $E_{\bi} = S_{\bi}(E)$ a \emph{cylinder} of $E$ of rank $k$.

For two words $\bi, \bj$, we use $\bi*\bj$ to denote the concatenation of $\bi$ and $\bj$,
and use $\bi\wedge\bj$ to denote the maximal common prefix of $\bi$ and $\bj$,
and denote by $|\bi|$ the length of $\bi$. Let $\bw|_q=w_1\ast\dots \ast w_q$ be the prefix of $\bw$ with length $q$.
For $z,z'\in\mathbb{R}^2$, we use $d(z,z')$ to denote the Euclidean distance between points $z$ and $z'$.
We define the \emph{coding} map $\pi:\SD^\infty\rightarrow E$ given by
$
\pi(\bomega)= \bigcap_{k\geq 1} S_{\bomega|_k}([0,1]^2).
$

Let $(\bx, \by)=((x_i)_{i\geq 1}, (y_i)_{i\geq 1})\in\SD^\infty$.  For $k\ge 1$, we set
\begin{equation}\label{appro-square-1}
Q_k(\bx,\by)=\pi(\left \{(\bu,\bv)\in \SD^\infty;  |\bx\wedge\bu|\geq k \text{ and } |\by\wedge\bv|\geq \ell(k)\right \}),
\end{equation}
and call it the \emph{approximate square} of $E$ of rank $k$, or \emph{ $k$-th approximate square} of $E$.


\begin{lemma}[\cite{Bed84,M84}]\label{lem:offspring} If $Q_k(\bx,\by)$  is a $k$-th approximate square of $E$,
then
$$\mu_E(Q_k(\bx,\by))= \dfrac{\prod_{j=k+1}^{\ell(k)} a_{y_j}}{N^{\ell(k)}}, ~~\text{or}~ \mu_E(Q_k(\bx,\by))=\dfrac{1}{N^k}~~ \text{if}~ \ell(k)=k.$$
\end{lemma}


\begin{lemma}\label{lem_esti}
Let $\bomega\in \SD^\infty$ and let $k\geq 1$, then
$$\dfrac{\mu_E(Q_{k+1}(\bomega))}{\mu_E(Q_{k}(\bomega))}\geq C_0^{-1}$$
where $C_0={nN^{1+1/\sigma}}$.
 \end{lemma}
 \begin{proof}
By Lemma \ref{lem:offspring}, we have
$$
\dfrac{\mu_E(Q_{k+1}(\bomega))}{\mu_E(Q_{k}(\bomega))}\geq \dfrac{1}{n N^{\ell(k+1)-\ell(k)}}\geq \dfrac{1}{nN^{1+1/\sigma}},
$$
and the lemma is proved.
\end{proof}

\section{\textbf{Lemmas} }\label{mainresult1}

In this section, we always assume that $E$ is a BM-carpet  of non-doubling type and satisfies $a_0>a_{m-1}$.
 Let $\mu$ be the  uniform Bernoulli measure on $E$,  let $0<\rho<n^{-3}$ and let $\varphi$
be a gauge function.

Let   $z\in E$ and let $\bomega=(\bx,\by)$ be a coding of $z$.
Let $0<r<1$.
Denote
\begin{equation}\label{def_U(z)}
U(z; r,\rho)=\sup_{\substack {z'\in E, B_{\rho r}(z')\subset B_{r}(z)}}\frac{\mu(B_{r}(z))}{\mu(B_{\rho r}(z'))},
\end{equation}
then $ \overline{\delta}_\varphi(z; \mu,\rho)=\underset{r\rightarrow0}{\limsup}\dfrac{\log U(z; r,\rho)}{\varphi(r)},$
so does $\underline{\delta}_\varphi(z; \mu,\rho)$.

In this section, we will always let $k(r)$ be the integer such that
\begin{equation}\label{eq:kr}
 \dfrac{1}{n^{k(r)+2}}< {r} \leq\dfrac{1}{n^{k(r)+1}}.
 \end{equation}
  Let
\begin{equation}\label{def_Xi}
\Xi_{k(r)}(z)=\{Q_{k(r)}(\bu,\bv);~ ~Q_{k(r)}(\bu,\bv)\cap B_r(z)\neq \emptyset\}.
\end{equation}
Since   $2r<n^{-k(r)}$,  we conclude   $\Xi_{k(r)}(z)$ contains at most four elements, and they locate
in one row,  or two adjacent rows of rank $k(r)$.

\begin{lemma}\label{lem:pull-back} There is a constant $C_1>1$ such that
\begin{equation}\label{assertion1}
U(z; r,\rho)\leq C_1\dfrac{\max \{ \mu(Q);~Q\in \Xi_{k(r)}(z)\}}
{\min \{\mu(Q);~Q\in \Xi_{k(r)}(z)\}}
\end{equation}
 and
\begin{equation}\label{assertion2}
  U(z; n^3r, \rho)\geq C_1^{-1}\dfrac{\max \{ \mu(Q);~Q\in \Xi_{k(r)}(z)\}}
{\min \{\mu(Q);~Q\in \Xi_{k(r)}(z)\}}.
\end{equation}
\end{lemma}

\begin{proof} For simplicity, we write $k:=k(r)$.

(i) Let $k_0$ be the integer such that $\dfrac{1}{n^{k_0+1}}<\rho\le\dfrac{1}{n^{k_0}}$.
Pick $z'\in E$ such that $B_{\rho r}(z')\subset B_r(z)$
and let $\bomega'$ be a coding of $z'$.
Since $\text{diam}(Q_{k+k_0+4}(\bomega'))<\dfrac{\sqrt{m^2+1}}{n^{k+k_0+4}}<\rho r$,
we have $Q_{k+k_0+4}(\bomega')\subset B_{\rho r}(z')$,  then by Lemma \ref{lem_esti} we have
\begin{equation}\label{eq_2}
\mu(B_{\rho r}(z'))\geq\mu(Q_{k+k_0+4}(\bomega'))\geq \dfrac{1}{C_0^{k_0+4}} \mu(Q_k(\bomega'))\ge \dfrac{1}{C_0^{k_0+4}}\min \{\mu(Q);~Q\in \Xi_{k}(z)\},
\end{equation}
where $C_0$ is the constant in Lemma \ref{lem_esti} and
 the last inequality holds by $Q_k(\bomega')\in \Xi_{k}(z)$
since $B_{\rho r}(z')\subset B_r(z)\subset \underset{Q\in \Xi_{k}(z)}{\bigcup}Q$.

On the other hand, $\mu(B_r(z))\leq 4\max \{ \mu(Q);~Q\in \Xi_{k}(z)\}$, and the first assertion holds
by setting $C_1=4C_0^{k_0+4}$.

 (ii)
 Clearly, $B_{n^3r}(z)$ contains all elements of $\Xi_k(z)$ since $n^3r>r+\text{diam~}Q$ for each $Q\in \Xi_k(z)$
 (we remark that it holds no matter $\Xi_k(z)$ contains  approximate squares of rank $k$ of two adjacent rows or the same row),
 then
$$\mu(B_{n^3r}(z))\geq  \max \{ \mu(Q);~Q\in \Xi_k(z)\}.$$
Suppose $Q_k(\bomega')$ attains the minimal measure in $\Xi_k(z)$.
 Since $a_1a_{m-1}\neq 0$, there exist $(i_1,0), (i_2, m-1)\in \SD$.
 Let $z'$ be the point with the coding
$$ (\bomega'|_{\ell(k)})\binom{i_1}{0} \binom{i_2}{m-1}^\infty .$$
Since $\rho n^3 r <n^{-(k+1)}$,  we have that $B_{\rho n^3 r}(z')\cap E$ is covered
by at most two  approximate squares of rank $k$ located in the same row as  $Q_k(\bomega')$, hence
 $$\mu(B_{\rho n^3 r}(z'))\leq 2 \min \{ \mu(Q);~Q\in \Xi_k(z)\},$$
 and the second assertion holds.
\end{proof}

Hereafter, we always assume that $C_1$ is the constant in Lemma \ref{lem:pull-back}.

\subsection{Reverse run length function}
Let $\SE=\{j;~a_{j}>0\}$.  Denote $a\vee b=\max\{a,b\}$.
 Now we define a kind of   reverse run length function for a sequence over $\SD$.
 Let $p\in \{0,m-1\}$ and $\bomega=(\bx,\by)\in \SD^\infty$. For $k\geq 1$, let $h_p(k;\bomega)$ be the maximal integer $h$  such that  $h\leq \ell(k)$  and $y_h\neq p$. We  define
\begin{equation*}
  \beta_0(k;\bomega)=\left \{
\begin{array}{ll}
 \ell(k)-(k\vee h_0(k;\bomega)), &
 \begin{array}{l} \text{if
 $y_{h_0(k;\bomega)}  \in \SE+1$,  }
 \end{array} \\
0, &\text{ otherwise (including  $h_0(k;\bomega)=0$);}
\end{array}
\right .
\end{equation*}
\begin{equation*}
  \beta_{m-1}(k;\bomega)=\left \{
\begin{array}{ll}
\ell(k)- (k\vee h_{m-1}(k;\bomega)), &
 \begin{array}{l}  \text{if $y_{h_{m-1}(k;\bomega)} \in \SE-1$,}
 \end{array} \\
0, &\text{ otherwise  (including  $h_{m-1}(k;\bomega)=0$).}
\end{array}
\right .
\end{equation*}
Define
\begin{equation}\label{eq:beta}
\beta(k; \bomega)=\beta_0(k;\bomega) \vee \beta_{m-1}(k;\bomega).
\end{equation}


\begin{lemma}\label{lem:upperbound}  There exists a constant $C_2>1$ such that for any $z=\pi(\bomega)\in E$ and $r\in(0,1)$,
\begin{equation}\label{esti_U(z)}
   U(z; r,\rho)\leq C_2\left (\frac{a_0}{a_{m-1}} \right )^{\beta(k(r);\bomega)};
\end{equation}
if $\beta(k(r);\bomega)\leq \ell(k(r))-k(r)-1$, then
\begin{equation}\label{esti_U}
   U(z;n^6 r,\rho)\geq C_2^{-1}\left (\frac{a_0}{a_{m-1}} \right )^{\beta(k(r);\bomega)}.
\end{equation}
\end{lemma}

\begin{proof} For simplicity, we write $k:=k(r)$.

We first prove \eqref{esti_U(z)}.  Clearly $Q_k(\bomega)=Q_k(\bx,\by)\in \Xi_k(z)$.
If all elements of $\Xi_k(z)$ are located in the same row as $Q_{k}(\bx,\by)$, then  \eqref{esti_U(z)} holds by
 the first assertion of Lemma \ref{lem:pull-back}.
So in the following we assume that there exists  an approximate square
  $Q_{k}(\bu,\bv)\in \Xi_k(z)$ which locates in a  row adjacent to the row of $Q_{k}(\bx,\by)$, in other words,  $\bv|_{\ell(k)}$ is adjacent to $\by|_{\ell(k)}$ in the lexicographical order.

\medskip

\textit{Case 1. }  $y_{\ell(k)}\not\in \{0, m-1\}$.

\medskip
 In this case $\beta(k;\bomega)=0$.
 Since $\bv|_{\ell(k)}$ and $\by|_{\ell(k)}$ are adjacent in the lexicographic order, we deduce that
they differ only at the last letter.
It follows that
\begin{equation}\label{eq:middle}
\frac{1}{n} \leq \dfrac{\mu(Q_k(\bx,\by))}{\mu(Q_k(\bu,\bv))}\le n.
\end{equation}
By Lemma \ref{lem:pull-back}, we have $U(z; r,\rho)\leq C_1n$.
\eqref{esti_U(z)}  holds in this case if we set $C_2=C_1 n$.

\medskip

\textit{Case 2. }  $y_{\ell(k)}=0$ or $m-1$.

\medskip
Without loss of generality, we assume that  $y_{\ell(k)}=0$.

If $\by\prec \bv$ in the lexicographical order, then
$\bv|_{\ell(k)}=\by|_{\ell(k)-1}*1$, so \eqref{eq:middle} holds, which implies  \eqref{esti_U(z)}  holds
as we did in Case 1.

If $\bv\prec \by$, then
there exist $h$ and $j$ such that
$$
\by|_{\ell(k)}=y_1\dots y_{h-1}*(j+1)*0^{\ell(k)-h},
\quad \bv|_{\ell(k)}=y_1\dots y_{h-1}*j*(m-1)^{\ell(k)-h}.
$$
Since $\beta(k;\bomega)=0$ or $\ell(k)-k \vee h$, we deduce that
$$ \frac{1}{n} \leq  \dfrac{\mu(Q_k(\bx,\by))}{\mu(Q_k(\bu,\bv))}\leq n\left (\frac{a_0}{a_{m-1}} \right )^{\beta(k;\bomega)}.
$$
Therefore, by Lemma \ref{lem:pull-back}, no matter $\mu(Q_k(\bx,\by))\geq \mu(Q_k(\bu,\bv))$ or not, we have
$$
 U(z; r,\rho)\leq C_1n \left (\frac{a_0}{a_{m-1}} \right )^{\beta(k;\bomega)},
 $$
and \eqref{esti_U(z)} also holds in this case if we set $C_2=C_1 n$.

Next, we prove \eqref{esti_U}. Suppose  $\beta(k;\bomega)\leq \ell(k)-k-1$.

If $\beta(k;\bomega)\leq 4/\sigma$, then \eqref{esti_U} holds if we set
$C_2\geq (a_0/a_{m-1})^{4/\sigma}$.

Hence, in the following we assume without loss of generality that  $y_{\ell(k)}=0$ and $\beta(k;\bomega)>4/\sigma$.
At this time $\beta_{m-1}(k;\bomega)=0$ and $\beta_0(k;\bomega)=\beta(k;\bomega)\leq \ell(k)-(k+1)$,
then $h:=h_0(k;\bomega)\geq k+1$.
 By the definition of $\beta(k;\bomega)$, there exists $i'$ such that
$(i',y_{h}-1)\in \SD$. Since $a_{m-1}>0$, there exists $i''$ such that $(i'',m-1)\in \SD$.
We define $\bomega'=(\bu,\bv)$  as
$$\bomega'=\bomega|_{h-1}*\binom{i'}{y_{h}-1}*\binom{i''}{m-1}^\infty:=\binom{\bx'}{\by'}.$$
 Let $z''=\pi(\bomega')$. Since $z$ and $z''$ belong to a same cylinder of $E$ of rank $h-1$ and $h-1\geq k$,  we obtain
  $d(z,z'')\leq \sqrt{\frac{1}{n^{2k}}+\frac{m^2}{n^{2k}}}<\frac{1}{n^{k-1}},$
 so $z''\in B_{n^3r}(z)$.

 We shall apply Lemma \ref{lem:pull-back} to $B_{n^3r}(z)$ instead of  $B_r(z)$.
 Notice that $k(n^3r)=k(r)-3=k-3$, we define $\Xi'_{k-3}(z)$ to be the collection of approximates squares
 of rank $(k-3)$ intersecting $B_{n^3r}(z)$.  Clearly
$Q_{k-3}(\bomega), Q_{k-3}(\bomega')\in \Xi'_{k-3}(z)$,
moreover, $\beta(k;\bomega)>4/\sigma$ implies $h<\ell(k-3)$, then we have $y_{\ell(k-3)}=0$ and
$y'_{\ell(k-3)}=m-1$,  so $Q_{k-3}(\bomega)$ and $Q_{k-3}(\bomega')$ locate in different rows.
 By Lemma \ref{lem:offspring},
$$   \dfrac{\mu(Q_{k-3}(\bomega))}{\mu(Q_{k-3}(\bomega'))}\geq \dfrac{1}{n}\left (\frac{a_0}{a_{m-1}} \right )^{\beta(k-3;\bomega)}\geq \dfrac{1}{n}\left (\frac{a_0}{a_{m-1}} \right )^{-3/\sigma-1} \cdot \left (\frac{a_0}{a_{m-1}} \right )^{\beta(k;\bomega)},
$$
 which together with the second assertion of Lemma \ref{lem:pull-back} imply \eqref{esti_U}
 if we set $C_2=C_1n({a_0}/{a_{m-1}})^{3/\sigma+1}.$

 The lemma is proved for all cases by setting $C_2=C_1 n({a_0}/{a_{m-1}})^{4/\sigma}.$
\end{proof}


\begin{cor}\label{cor:upper} It holds that
\begin{equation}\label{eq:upper-upper}
 \overline{\delta}_\varphi(z; \mu,\rho)\leq \limsup_{k\to \infty} \frac{{\beta}(k; {\bomega})}{\varphi(n^{-k})}\cdot\log ({a_0}/a_{m-1}).
\end{equation}
\end{cor}

\begin{proof}Since $k(r)$ runs over all integers as $r$ runs over $(0,1)$, we have
$$
\begin{array}{rl}
  \overline{\delta}_\varphi(z; \mu,\rho) &=   \underset{r\to 0}{\limsup} ~\dfrac{\log U(z; r,\rho)}{\varphi(r)} \\
 & \leq \underset{k\to \infty}{\limsup} ~\dfrac{\log C_2+ {\beta(k;\bomega)} \log \left (\dfrac{a_0}{a_{m-1}} \right )}{\varphi  (n^{-k-1)})}\\
 &\leq \underset{k\to \infty}{\limsup} ~\dfrac{{\beta}(k; {\bomega})}{\varphi(n^{-k})}\cdot \log ({a_0}/a_{m-1}),
 \end{array}
$$
 where the  last inequality holds  since $\varphi$ is decreasing, and the corollary is proved.
\end{proof}

\section{\textbf{Proofs of Theorem \ref{thm:lower} and Theorem \ref{thm:formula}} }\label{proved Main1-2}
In this section, we assume the same assumptions on $E=K(n, m, \SD)$ as Section \ref{mainresult1}.
\begin{lemma}\label{lem:upper-bound}
Let $\varphi$ be a gauge function satisfying
\begin{equation}
\lim_{k\to \infty} \frac{k}{\varphi(n^{-k})}=s\in (0,+\infty].
\end{equation}
  Let $z=\pi(\bomega)\in E$. If $\bomega\in \Omega_0\cup \Omega_{m-1}$, then
$$
{\delta}_\varphi(z; \mu,\rho)=s(1/\sigma-1)\log(a_0/a_{m-1})\cdot {\mathbf 1}_{V_E}(z).
$$
\end{lemma}

\begin{proof}
\textit{Case 1.}
   $z\not\in V_E$.

 \medskip

Without loss of generality, let us assume that  $\by$ ends with $0^\infty$.
Write $\by=y_1\cdots y_{n_0}0^\infty$ where $y_{n_0}\neq0$ for some $n_0\in\mathbb{N}$.

Let $L$ be the horizontal line containing $z$.
Since $z$ is not a double vertical coding point,  any cylinder of $E$ of rank $n_0$ below $L$ does not contain $z$.
Let $K$ be the union  of cylinders of $E$ of rank $n_0$ below $L$, then $\dist(z,K)>0$.

Let  $r< \dist(z,K)$.
Let $k=k(r)$ be defined in \eqref{eq:kr}.  Furthermore we choose $r$ small so that  $k>n_0$,
then $z$ locates at the bottom of $Q_{k}(\bomega)$,  so $\Xi_k(z)$ consists of
 at most two approximate squares of rank $k$ in the same row, and one of them  is $Q_{k}(\bomega)$.
It follows that $U(z; r, \rho)\leq C_1n$ by  Lemma \ref{lem:pull-back}. Therefore ${\delta}_\varphi(z; \mu,\rho)=0$.

\medskip
\textit{Case 2.}
$z\in V_E$.
\medskip

Let $\bomega'=(\bx',\by')$ be another coding of $z$ such that $\by'\neq \by$.
Without loss of generality, let us assume that
$$\by=y_1\dots y_{n_0}(j+1)0^\infty, \quad \by'=y_1\dots y_{n_0}j(m-1)^\infty,$$
for some $n_0\in \mathbb{N}$.
Pick $r>0$ and let $k=k(r)$ be  defined in \eqref{eq:kr}.
We choose $r$ small so that  $k>n_0$.
 Notice that  $Q_k(\bx,\by),Q_k(\bx',\by')\in \Xi_k(z)$, and $\beta(k;\bomega)=\ell(k)-k$.
 By Lemma \ref{lem:pull-back} we have
$$
U(z;n^3r,\rho)\geq C_1^{-1} \left (\dfrac{a_0}{a_{m-1}}\right )^{\beta(k;\bomega)},
$$
thus, using $\varphi(n^3 r)\leq \varphi(n^{-k})$, we have
$$
\underline{\delta}_\varphi(z; \mu,\rho)\geq \lim_{k\to \infty}  \frac{\ell(k)-k}{\varphi(n^{-k})}\log  ({a_0}/a_{m-1})=s(1/\sigma-1)\log  ({a_0}/a_{m-1}).
$$
 The other direction inequality is due to Corollary \ref{cor:upper}.
 The lemma is proved.
\end{proof}

\begin{proof}[\textbf{Proof of Theorem \ref{thm:lower}.}]
  Pick $z\in E$. Let $\bomega=(\bx,\by)$ be a coding of $z$.
  If $\bomega\in \Omega_0\cup \Omega_{m-1}$, the theorem holds by
Lemma \ref{lem:upper-bound}.

 Now we assume that $\bomega\not\in \Omega_0\cup \Omega_{m-1}$. Then
  $\by$ does not end in $0^\infty$ and $(m-1)^\infty$.
Let $(y_{j_t})_{t\geq 1}$ be a subsequence of  $\by=(y_j)_{j\geq 1}$
satisfying that
$$(y_{j_t}, y_{j_t+1})\neq   (0, 0) \text{ and } (m-1,m-1), \ \ t\geq 1.$$
 Let $k_t$ be the smallest integer such that $j_t+1\leq \ell(k_t)$, then $\beta(k_t;\bomega)\leq 1/\sigma+2$.

Set $r_t=n^{-k_t-1}$.  By the first assertion of Lemma \ref{lem:upperbound}, we have
$$\underline{\delta}_\varphi(z; \mu,\rho)\leq \liminf_{t\to \infty} \dfrac{\log C_2+(1/\sigma+2)\log(a_0/a_{m-1})}{\varphi(n^{-k_t-1})}  =0.$$
The theorem is proved.
\end{proof}


\begin{proof}[\textbf{Proof of Theorem \ref{thm:formula}}]
 Item (ii) is a consequence of Lemma \ref{lem:upper-bound}.
 In the following, we prove Item (i). Moreover, by Corollary \ref{cor:upper}, we only need to show that
\begin{equation}\label{eq:Omega}
 \overline{\delta}_\varphi(z; \mu,\rho)\geq  \limsup_{k\to \infty} \frac{{\beta}(k; {\bomega})}{\varphi(n^{-k})}\log (a_0/a_{m-1}), ~~\text{for}~~\bomega\not\in \Omega_0\cup\Omega_{m-1}.
\end{equation}

Let $z=\pi(\bomega)\in E$ with $\bomega\not\in \Omega_0\cup\Omega_{m-1}$.

Suppose $\beta(k;\bomega)= \ell(k)-k$ holds for infinitely many $k$. Let $h$ be the largest integer such that $h\leq k-1$ and
$y_{h+1}\neq 0$. Then $\beta(h;\bomega)=\ell(h)-h-1$. Therefore, there exists a subsequence $(k_j)_{j\geq 1}$ of integers   such that
 $\beta(k_j;\bomega)= \ell(k_j)-k_j-1$.
set $r_j=1/n^{k_j+1}$. By the second assertion of Lemma \ref{lem:upperbound}, we have
$$
\begin{array}{rl}
\overline\delta_\varphi(z; \mu,\rho) & \geq \underset{j\to \infty}{\limsup}\dfrac{\log U(z; n^6 r_j, \rho)}{\varphi(n^6 r_j)}
\geq \underset{j\to \infty}{\limsup}\dfrac{\log C_2^{-1}+{\beta}(k_j; {\bomega})\log (a_0/a_{m-1})}{\varphi(n^{-k_j+5})} \\
& =\underset{j\to \infty}{\limsup}\dfrac{\ell(k_j)-k_j-1}{\varphi(n^{-k_j+5})}\log (a_0/a_{m-1})
=s(1/\sigma-1)\log (a_0/a_{m-1}),
\end{array}
$$
so the lemma holds in this case since $\underset{k\to \infty}{\limsup} \frac{{\beta}(k; {\bomega})}{\varphi(n^{-k})}\leq s(1/\sigma-1)$.

Suppose there exists $k_0\in\mathbb{N}$ such that $\beta(k;\bomega)\leq \ell(k)-k-1$ holds for $k\geq k_0$.
Set $r_k=1/n^{k+1}$. Similar as above,  applying  Lemma \ref{lem:upperbound} to $k\geq k_0$ instead of $k_j$, we have
$$
\begin{array}{rl}
\overline\delta_\varphi(z; \mu,\rho) 
 \geq  \underset{k\to \infty}{\limsup}  \dfrac{ {\beta}(k; {\bomega})}{\varphi(n^{-k+5})} \log (a_0/a_{m-1})
\geq \underset{k\to \infty}{\limsup}  \dfrac{ {\beta}(k; {\bomega})}{\varphi(n^{-k})}\log (a_0/a_{m-1}).
\end{array}
$$
(The last equality holds since  $\varphi$ is decreasing.)
This proves \eqref{eq:Omega} as well as the item (i).
\end{proof}

\section{\textbf{Remarks on $\delta_{-\log r}$   and $\delta_{\log|\log r|}$}}\label{mainR2_Coro1}

  Recall that $\delta=\delta_{-\log r}$ and $\Delta=\delta_{\log|\log r|}$ (See Example \ref{exam-1}).

\subsection{Run length}
Let $\theta\in \{0,1,\dots, m-1\}$.  Set $\bomega=(\bx,\by)\in \SD^\infty$ with $\by=(y_j)_{j\geq 1}$, and let $l_\theta(k; \by)$ be the run length of  the letter $\theta$ in $\by$ at the position $k$, \textit{i.e.}, $l_\theta(k; \by)=t$ if $y_k=\cdots=y_{k+t-1}=\theta$ and
$y_{k+t}\neq \theta$. By convention we set  $l_\theta(k; \by)=0$ if $y_k\neq\theta$.
It is well known that

\begin{lemma}[Billingsley \cite{Bill94}]\label{Bill_Lemma}
Let $\Sigma=\{0,1,\dots, m-1\}$. Let $\bp=(p_0,\dots, p_{m-1})$ be a probability weight which allows
$p_j=0$. Let $\nu_{\bp}$ be the  Bernoulli measure on $\Sigma^\infty$ associated with $\bp$, then
for all $p_\theta\neq 0$,
$$
{\nu_{\bp}}\left(\left \{  \by\in \Sigma^\infty; ~\underset{k\rightarrow \infty}{\limsup} \dfrac{l_\theta(k;\by)}{-\log_{p_\theta} k}=1\right \}  \right)=1.
$$
\end{lemma}

\begin{defi}[Modified run length]\emph{ Let $\theta_0\in \Sigma$ be a fixed letter.
Define $l_{\theta_0,\theta}(k,\by)=t$ if $y_{k-1}=\theta_0$, $y_k=\cdots=y_{k+t-1}=\theta$ and
$y_{k+t}\neq \theta$; otherwise, set $l_{\theta_0,\theta}(k,\by)=0$. }
\end{defi}

By the same argument as the proof of Lemma \ref{Bill_Lemma}, one can show that

\begin{lemma}\label{Modified_Bill_Le}
Let $\Sigma=\{0,1,\dots, m-1\}$. Let $\bp=(p_0,\dots, p_{m-1})$ be a probability weight which allows
$p_j=0$. Let $\nu_{\bp}$ be the  Bernoulli measure on $\Sigma^\infty$ associated with $\bp$.
If $p_{\theta_0}\neq 0$ and  $p_\theta\neq 0$, then
$$
{\nu_{\bp}}\left(\left \{  \by\in \Sigma^\infty; ~\underset{k\rightarrow \infty}{\limsup} \dfrac{l_{\theta_0,\theta}(k;\by)}{-\log_{p_\theta} k}=1\right \}  \right)=1.
$$
\end{lemma}

Let $\mathbb P$ be the uniform Bernoulli measure on $\SD^\infty$.

\begin{lemma}\label{lem:run}
Let $E=K(n, m, \SD)$ be a BM-carpet of non-doubling type. Assume that $a_0>a_{m-1}$.  Then
$$
{\mathbb P}\left(\left \{  \bomega \in \SD^\infty; ~\underset{k\rightarrow \infty}{\limsup} \dfrac{\beta(k;\bomega)}{-\log_{p_0} k/\sigma}=1\right \}  \right)=1,
$$
where $p_0=a_0/\#\SD$.
\end{lemma}

\begin{proof}  Pick  $\theta_0$ such that $\theta_0, \theta_0-1\in \SE$.  Let $L\geq 1$ be an integer and set $\epsilon=1/L$.
 Denote
$$
H_{L,1}=\left[ (\bx,\by)\in\SD^\infty: l_0(k;\by)+l_{m-1}(k;\by)< (1+\epsilon) \log k/(-\log p_0) \text{ eventually} \right],
$$
$$
H_{L,2}=\left[ (\bx,\by)\in\SD^\infty:  l_{\theta_0, 0}(k;\by)\geq (1-\epsilon)\log k/(-\log p_0) \text{ infinitely often} \right].
$$
 Since $p_0>p_{m-1}$, by Lemma \ref{Bill_Lemma} we have ${\mathbb P}(H_{L,1})=1$.
 On the other hand, we have ${\mathbb P}(H_{L,2})=1$ by Lemma \ref{Modified_Bill_Le}.
 Let $H_L=H_{L,1}\cap H_{L,2}$, then ${\mathbb P}(H_L)=1$.

  Fix $\bomega=(\bx,\by)\in H_L$. For $k$ large enough,  we have
 $$
 \beta(k;\bomega)\leq  l_0(\ell(k)-\beta(k;\bomega)+1;\by)+l_{m-1}(\ell(k)-\beta(k;\bomega)+1;\by)\leq  (1+\epsilon)\log \ell(k)/(-\log p_0)
 $$
  eventually since $\bomega\in H_{L,1}$.
 It follows that
\begin{equation}\label{esti_beta1}
\underset{k\rightarrow \infty}{\limsup} \dfrac{\beta(k;\bomega)}{-\log_{p_0} k/\sigma}\leq 1+\epsilon.
\end{equation}
On the other hand, by $\bomega\in H_{L,2}$, there exists infinite many $h$ such that
$$ l_{\theta_0, 0}(h;\by)\geq (1-\epsilon)\log h/(-\log p_0).$$
Set $k= \sigma(h-1+(1-\epsilon)(-\log_{p_0} h ))$, we obtain $\beta(k;\bomega)\geq (1-\epsilon)\log h/(-\log p_0))-1$.
It follows that
\begin{equation}\label{esti_beta2}
\underset{k\rightarrow \infty}{\limsup} \dfrac{\beta(k;\bomega)}{-\log_{p_0} k/\sigma}\geq 1-\epsilon.
\end{equation}
Set $H=\bigcap_{L\geq 1} H_L$, then $H$ has full measure, and  each $\bomega\in H$,
$$\underset{k\rightarrow \infty}{\limsup} \dfrac{\beta(k;\bomega)}{-\log_{p_0} k/\sigma}= 1.$$
The lemma is proved.
\end{proof}

\subsection{Remarks on $\delta$ and $\Delta$}
 \begin{thm}\label{thm:delta} Let $E=K(n,m,\SD)$ be a BM-carpet of non-doubling type. Assume that $a_0>a_{m-1}$. Then
 for $\mu_E$-a.e. $z\in E$, ${\delta}(z; \mu_E)=0$, and the range of $\overline{\delta}(z; \mu_E)$ is $[0,  \delta_{\max}]$.
 \end{thm}

\begin{proof} By Lemma \ref{lem:run},
$$
\underset{k\rightarrow \infty}{\limsup} \dfrac{\beta(k;\bomega)}{k}=0
$$
for ${\mathbb P}$-a.e. $\bomega\in \SD^\infty$, hence ${\delta}(z; \mu_E)=0$ for $\mu_E$-a.e. $z\in E$ by \eqref{formula_delta} in Example \ref{exam-1} and Theorem \ref{thm:lower}.

Next, we prove that
 $\{\overline{\delta}(z; \mu_E);~z\in E\}=[0, \delta_{\max}].$
  Pick $t\in [0, \delta_{\max}]$, let
$t'= {t}/{\log_n(a_0/a_{m-1})}.$
Since $E$  is of non-doubling type,
 there exist $(i',j),(i,j+1)\in\SD$ and $(i_0,0)\in \SD$.

Take an integer $p_1>\sigma/(1-\sigma)$, and set
$p_{k+1}=\ell(p_k)$ for  $k\geq 1$. Then $p_k\to \infty$ as $k \to \infty$.
Set
\begin{equation}\label{eq:curve}
\bomega=
\left(\begin{array}{c}
i_0\\
0
\end{array}\right)^{p_1}
\prod_{k=1}^\infty
\left[\left(\begin{array}{c}
i\\
j+1
\end{array}\right)^{\ell(p_k)-p_k-\lfloor t'p_k \rfloor}
\left(\begin{array}{c}
i_0\\
0
\end{array}\right)^{\lfloor t'p_k \rfloor}\right],
\end{equation}
and let $z=\pi(\bomega)$.
Then for $k\geq p_1$,
$$\beta(k;\bomega)\leq t'k  \quad \text{ and } \quad
\beta(p_k;\bomega)=\lfloor t'p_k \rfloor.$$
 So by \eqref{formula_delta}, we have
\begin{equation}\label{upp_delta}
\bar{\delta}(z; \mu_E)=\left(\limsup_{k\to \infty} \dfrac{\beta(k;\bomega)}{k}\right)\log_n(a_0/a_{m-1})=t'\log_n(a_0/a_{m-1})=t.
\end{equation}
\end{proof}

 \begin{thm}\label{thm:Delta} Let $E=K(n,m,\SD)$ be a BM-carpet of non-doubling type. Assume that $a_0>a_{m-1}$. Then for $\mu_E$-a.e. $z\in E$,
\begin{equation}\label{eq_full_Mea1}
 \overline{\Delta}(z; \mu_E,\rho)=- \log_{p_0}(a_0/a_{m-1})
\end{equation}
where $p_0=a_0/\#\SD$.
\end{thm}

\begin{proof} Clearly $\mu_E(\{z\in E; z=\pi(\bomega) ~~~\text{with} ~~\bomega\in\Omega_0\cup \Omega_{m-1} \})=0$.
 So the theorem  is a direct consequence of Lemma \ref{lem:run} and \eqref{formula_Delta}.
\end{proof}

\section{\textbf{A   point-wise doubling index $\gamma$}}\label{mainresult2}

In this section, we define an alternative  point-wise doubling index.

Let $\mu$ be a measure on a metric space $X$. Fix $0<\rho<1$. For $z\in X$, the \emph{upper doubling index} of $\mu$ at $z$ is defined by
 $$
 \overline{\gamma}(z; \mu,\rho)=\limsup_{r\rightarrow0}\sup_{\substack {B_{\rho r}(z')\subset B_{r}(z)}}\frac{\log \mu(B_{\rho r}(z'))}{\log \mu(B_{r}(z))};
 $$
 similarly, we define  $\underline{\gamma}(z; \mu,\rho)$.
 If the two values agree,  we denote it by
  $\gamma(z; \mu,\rho)$.

\medskip

Let   $E=K(n, m, \SD)$ be a BM-carpet,
and  $\mu$ be the uniform Bernoulli measure on $E$.
Denote $N=\#\SD$, and denote
$$\gamma_{\max}=\dfrac{\log N-(1-\sigma)\log a_{m-1}}{\log N-(1-\sigma)\log a_{0}}.$$

\begin{thm}\label{thm:gamma1} Let $E=K(n, m, \SD)$ be a BM-carpet of non-doubling type. Assume that $a_0>a_{m-1}$.
Let $0<\rho<n^{-3}$. Then $\overline{\gamma}(z; \mu,\rho)\in [1, \gamma_{\max}]$ and $\gamma_{\max}$
is attainable.
\end{thm}

\begin{proof}
Let $z\in E$ and let $\bomega$ be a coding of $z$.
Pick $0<r<1$.
Let $k=k(r)$ be the integer such that
$\dfrac{1}{n^{k(r)+2}}< {r} \leq\dfrac{1}{n^{k(r)+1}}$
and let
\begin{equation}\label{def_Xi}
\Xi_k(z):=\{Q_k(\bu,\bv);~ ~Q_k(\bu,\bv)\cap B_r(z)\neq \emptyset\}
\end{equation}
as in Section \ref{mainresult1}.

Let $z'\in E$ and $B_{\rho r}(z')\subset B_r(z)$. By the first assertion of Lemma \ref{lem:upperbound}, we have
\begin{equation}\label{esti_gamma1}
 \dfrac{\log \mu(B_{\rho r}(z'))}{\log \mu(B_r(z))}
   \leq \dfrac{-\log U(z;r,\rho)}{\log \mu(B_r(z))} +1 \leq \dfrac{-\log C_2-\beta(k;\bomega)\log(a_0/a_{m-1})  }{ \log (4\cdot \max \{\mu(Q);~Q\in \Xi_k(z)\})}+1
\end{equation}
where the last inequality holds since $\Xi_k(z)$ is a cover of $B_r(z)$.

Let $Q_k(\bu,\bv)$ be an approximate square located in a row adjacent to the row of $Q_k(\bomega)$.
By analysing the three cases that $\bv|_{\ell(k)}$ equals to $\by|_{\ell(k)-1}*v_{\ell(k)}$,
  $ \by|_{h-1}*v_{h}*0^{\ell(k)-h}$, or $ \by|_{h-1}*v_{h}*(m-1)^{\ell(k)-h}$,
we have
$$
  \mu(Q_k(\bu,\bv)) \leq \dfrac{n(a^{*})^{\ell(k)-k-\beta(k;\bomega)}a_0^{\beta(k;\bomega)}}{N^{\ell(k)}},
$$
where $a^{*}=\underset{j\in \SE}{\max}~a_j$. It follows that
\begin{equation}
   \dfrac{\log \mu(B_{\rho r}(z'))}{\log \mu(B_r(z))}\leq
 \dfrac{-\log C_2+\log\left( \dfrac{a_{m-1}}{a_0}\right)^{\beta(k;\bomega)}}{\log (4n)+\log \dfrac{(a^{*})^{\ell(k)-k-\beta(k;\bomega)}a_0^{\beta(k;\bomega)}}{N^{\ell(k)}}}+1.
\end{equation}
It is easy to show that after removing the unimportant constants $-\log C_2$ and $\log (4n)$,
the right hand side of the above formula attains maximum when $\beta(k;\bomega)=\ell(k)-k$.
Hence
\begin{equation}
\overline{\gamma}(z; \mu,\rho)  \leq
     \limsup_{k\to \infty}  \left[ \log\left( \frac{a_{m-1}}{a_0}\right)^{\ell(k)-k} \left ( \log \dfrac{ a_0^{\ell(k)-k}}{N^{\ell(k)}}
     \right )^{-1}+1\right]=\gamma_{\max}.
\end{equation}

By the second assertion of  Lemma \ref{lem:upperbound}, if $\beta(k;\bomega)\leq \ell(k)-k-1$, we have
\begin{equation}\label{esti_gamma2}
\begin{array}{rl}
  \underset{z'\in E, B_{\rho n^6 r}(z')\subset B_{n^6r}(z)}{\sup}\dfrac{\log \mu(B_{\rho  n^6 r}(z'))}{\log \mu(B_{n^6r}(z))}
  &=  \dfrac{-\log U(z;n^6r,\rho)}{\log \mu(B_{n^6r}(z))}+1\\
  &\geq \dfrac{\log C_2-\beta(k;\bomega)\log(a_0/a_{m-1})}{\log (\max \{\mu(Q);~Q\in \Xi_k(z)\})}+1
 \end{array}
\end{equation}
where the last inequality holds  since $B_{n^3r}(z)$ covers every element of $\Xi_k(z)$.

Let $\bomega$ be a point in $\SD^\infty$ such that
$\beta(k; \bomega)=\ell(k)-k-1$ for infinitely many $k$.
For example, the $\bomega$ in \eqref{eq:curve} will do with a slight adjustment for $t'=1/\sigma-1$.
Let $z=\pi(\bomega)$. By  \eqref{esti_gamma2} we have
\begin{equation}
\overline{\gamma}(z; \mu,\rho)\geq
\underset{j\rightarrow\infty}{\lim}\left(\dfrac{\log C_2+\log\left( \frac{a_{m-1}}{a_0}\right)^{\ell(k_j)-k_j-1}}{\log (a^*)+\log \frac{a_0^{\ell(k_j)-k_j-1}}{N^{\ell(k_j)}}}+1\right)
  =\gamma_{\max}.
\end{equation}
The theorem is proved.
\end{proof}

\begin{remark} \emph{Under the assumption of Theorem \ref{thm:gamma1}, one can show that }

 \emph{(i) If $\bomega\in \Omega_0\cup \Omega_{m-1}$, then $
{\gamma}(z; \mu,\rho)=\max\{
\gamma_{\max}\cdot {\mathbf 1}_{V_E}(z), ~1\}.$}

\emph{(ii)  $\underline{\gamma}(z; \mu,\rho)= \max\{
\gamma_{\max}\cdot {\mathbf 1}_{V_E}(z), ~1\}.$}

\emph{(iii)   For $\mu$-a.e. $z\in E$, $\overline{\gamma}(z; \mu,\rho)=1$.}
\end{remark}

\section{\textbf{Proofs of Theorems \ref{thm_5}, \ref{thm:fiber} and \ref{thm:best}}}\label{mainresult3}

In this section, we always assume that $E=K(n,m, \SD), F=K(n,m,\SD')$ are two BM-carpets
and $f:E\rightarrow F$ is a bi-Lipschitz map with Lipschitz constant $c_0$.
We use $(a_j)_{j=0}^{m-1}$ and $(a'_j)_{j=0}^{m-1}$ to denote the fiber sequences
of $\SD$ and $\SD'$, respectively.
Denote $N=\#\SD$ and $N'=\#\SD'$.
Denote $s=\#\{j; ~a_j>0\}$ and $s'=\#\{j; ~a'_j>0\}$.
Let
\begin{equation}\label{defi_a*}
a_1^*>  a_2^*> \cdots>  a_{\tilde p}^*
\end{equation}
be  the distinct non-zero terms of   $(a_j)_{j=0}^{m-1}$, and let
 $M_i$  be  the occurrence of $a_i^*$, let
\begin{equation}\label{defi_b*}
b_1^* > b_2^*>\cdots>  b_{\tilde q}^*
\end{equation}
be the  distinct non-zero terms of     $(a'_j)_{j=0}^{m-1}$, and let
  $M'_i$  be the occurrence of $b_i^*$.

Denote by $\mu_E$ and $\mu_F$ the uniform Bernoulli measures of $E$ and $F$ respectively.
 Huang \emph{et al.}\cite{HRWX22} obtained the following Theorem, see also Rao \emph{et al.} \cite{Rao2019}.
\begin{thm}[\cite{HRWX22,Rao2019}]\label{thm:equivalent}
Let $E$ and $F$ be two BM-carpets.
If $f:~E\to F$ is a bi-Lipschitz map, then $\mu_F\circ f$ is equivalent to
$\mu_E$, namely, there exists $\zeta>0$ such that
\begin{equation}\label{eq:AC}
\zeta^{-1} \mu_E(A) \leq      \mu_F(f(A))\leq \zeta \mu_E(A)
\end{equation}
for any Borel set $A\subset E$.
 Consequently, $\mu_E$ and $\mu_F$ have the same multifractal spectrum.
\end{thm}

 By the above result, Rao \textit{et al.} \cite{Rao2019} characterized when $\mu_E$ and $\mu_F$ have the same multifractal spectrum, see also Banaji \emph{et al.} \cite{B}.

\begin{thm}[\cite{Rao2019,B}]\label{thm:invariants}
Let $E=K(n,m,\SD)$ and $F=K(n,m,\SD')$ be two BM-carpets.
Then $\mu_E$ and $\mu_F$ have the same multifractal spectrum if and only if
\begin{equation}\label{eq:invariants}
\tilde p=\tilde q \ \  \text{ and } \ \
\frac{a_i^*}{b_i^*}=\left (\frac{M_i'}{M_i}\right )^{1/\sigma}=\left (\frac{s'}{s}\right )^{1/\sigma}=\left(\frac{N}{N'}\right)^{1/(1-\sigma)},  ~~\text{for }~~ i=1,\dots, \tilde p.
\end{equation}
\end{thm}

\begin{proof}[\textbf{Proof of Theorem \ref{thm_5}}]
  Without loss of generality, we assume that $a_0>a_{m-1}>0$ and
 $a'_0>a'_{m-1}>0$.
Let $\zeta$ be the constant in Theorem \ref{thm:equivalent} with respect to $E$ and $F$.
 Let $f:E\to F$ be a bi-Lipschitz map with a Lipschitz constant $c_0$.

Firstly, we prove  \eqref{eq_MainThm3}.
Pick any $z\in E$. Let $0<r<1$.
 Then
$$f(B_{c_0^{-1}r}(z))\subset B_{r}(f(z))\subset f(B_{c_0r}(z)).$$
By Theorem \ref{thm:equivalent}, we have
\begin{equation}\label{sec4_1}
\zeta^{-1}\mu_{E}(B_{c_0^{-1}r}(z))\leq \mu_{F}(B_{r}(f(z)))\leq \zeta\mu_{E}(B_{c_0r}(z)).
\end{equation}
Let $0<\rho<n^{-3}c_0^{-2}$.  Pick any $\omega'\in F$ satisfying $B_{\rho r}(\omega')\subset B_{r}(f(z))$. Similar as \eqref{sec4_1}, we have
\begin{equation}\label{sec4_2}
\zeta^{-1}\mu_{E}(B_{c_0^{-1}\rho r}(f^{-1}(\omega')))\leq \mu_{F}(B_{\rho r}(\omega'))\leq \zeta\mu_{E}(B_{c_0\rho r}(f^{-1}(\omega'))).
\end{equation}
By \eqref{sec4_1} and \eqref{sec4_2}, we have
$$
\log\left(\frac{\mu_{F}(B_r(f(z)))}{\mu_{F}(B_{\rho r}(\omega'))}\right)\leq\log\zeta^2+\log \left(\frac{\mu_E(B_{c_0r}(z))}{\mu_E(B_{c_0^{-1}\rho r}(f^{-1}(\omega')))}\right).
$$
Since $B_{\rho r}(\omega')\subset B_{r}(f(z))$, we have $B_{c_0^{-1}\rho r}(f^{-1}(\omega'))\subset B_{c_0r}(z)$, so
$$
\overline{\delta}_\varphi(f(z);\mu_F,\rho) \leq  \overline{\delta}_\varphi(z;\mu_E,c_0^{-2}\rho);
$$
 by symmetry,  we have
 $\overline{\delta}_\varphi(z;\mu_E,\rho c_0^2)\leq \overline{\delta}_\varphi(f(z);\mu_F,\rho).$
Finally, since $\overline{\delta}_\varphi$
is irrelevant with $\rho$, we obtain
$$\overline{\delta}_\varphi(z;\mu_E )= \overline{\delta}_\varphi(f(z);\mu_F).$$
Clearly,  the above  equality also holds if we replace $\overline{\delta}_\varphi$
by $\underline{\delta}_\varphi$. \eqref{eq_MainThm3} is proved.

Let $\delta=\delta_{-\log r}$, then $\underline\delta(z;\mu_E)=\delta_{\max}\cdot {\mathbf 1}_{V_E}(z)$. Hence, $z\in V_E$ if and  only if  $\underline\delta(z;\mu_E)>0$,
if and only if  $\underline\delta(f(z);\mu_F)>0$,
if and only if $f(z)\in V_F$.
 This proves $V_F=f(V_E)$. In particular,  $V_E=\emptyset$ if and only if
$V_F=\emptyset$, that is, $E$ satisfies the VSC if and only if $F$ does.
\end{proof}

\begin{remark}\label{rema_gamma}\emph{
Under the assumption of Theorem \ref{thm_5}, we also have
\begin{equation}\label{equi_Delta}
\overline \gamma(z; \mu_E)= \overline \gamma(f(z); \mu_F).
\end{equation}
}
\end{remark}

\begin{proof}[\textbf{Proof of Theorem \ref{thm:fiber}}]
  Without loss of generality, we assume that $a_0>a_{m-1}>0$ and $a'_0>a'_{m-1}>0$.
  By Theorem \ref{thm:delta} and Theorem \ref{thm_5}, we have  $\delta_{E,{\max}}= \delta_{F,{\max}}$, then
\begin{equation}\label{gamma_delta}
\dfrac{a_0}{a_{m-1}}=\dfrac{a'_0}{a'_{m-1}}.
\end{equation}
Similarly, by Theorem \ref{thm:gamma1} and Remark \ref{rema_gamma} we have $\gamma_{E,{\max}}=\gamma_{F,{\max}}$,
thus
$$
 \dfrac{(1-\sigma)\log (a_{0}/a_{m-1})}{\log N-(1-\sigma)\log a_{0}}=   \dfrac{(1-\sigma)\log (a_{0}'/a_{m-1}')}{\log N'-(1-\sigma)\log a_{0}'}
$$
by using $\gamma_{E,{\max}}-1=\gamma_{F,{\max}}-1$,
summing up with \eqref{gamma_delta} implies that
\begin{equation}\label{eq_dist1}
\dfrac{a_0}{a_0'}=\dfrac{a_{m-1}}{a_{m-1}'}=\left(\dfrac{N}{N'}\right)^{\frac{1}{1-\sigma}}.
\end{equation}
By Theorem \ref{thm:Delta}, we have
$$
 ~\overline{\Delta}(z; \mu_E,\rho)=- \log_{p_0}(a_0/a_{m-1}) \text{ for $\mu_E$-a.e. $z\in E$,}
$$
$$
  ~\overline{\Delta}(f(z); \mu_F,\rho)=- \log_{p_0'}(a'_0/a'_{m-1}) \text{ for $\mu_F$-a.e. $f(z)\in F$,}
$$
where $p_0'=a'_0/{N'}$. So $\log_{p_0}(a_0/a_{m-1})=\log_{p_0'}(a'_0/a'_{m-1})$ by \eqref{eq_MainThm3} with $\overline{\Delta}=\overline{\delta}_{\log|\log r|}$ and it follows that
$$
\dfrac{a_0}{N}=p_0=p_0'=\dfrac{a'_0}{N'},
$$
 which together with \eqref{eq_dist1}  imply $N=N'$.
By \eqref{eq:invariants} and \eqref{eq_dist1}, the theorem is proved.
\end{proof}

\begin{proof}[\textbf{Proof of Theorem \ref{thm:best}}]
Let $E$ and $F$ be two BM-carpets 
 and $E\sim F$.

If $E$ and $F$ are of non-doubling type, the theorem holds by Theorem \ref{thm:fiber}.
If $E$ and $F$ do not possess vacant rows, then $s=s'=m$, thus $N=N'$ by \eqref{eq:invariants} in Theorem \ref{thm:invariants}, the theorem holds in this case.
Hence, in the following, we always assume that $E, F\in {\mathcal M}_{v,d}$.

If $E$ and $F$ are not totally disconnected, since both $E$ and $F$ possess vacant rows, all connected components of $E$ and $F$ are horizontal line segments of length $1$,
then $a_1^*=(a'_1)^*=n$, thus $N=N'$ by \eqref{eq:invariants} in Theorem \ref{thm:invariants}, and the theorem holds in this case.
This proves that the theorem holds provided $E,F\not\in {\mathcal M}_{t,v,d}$.

Finally, by \cite[Theorem 1.4]{Rao2019},  the theorem holds in the case that $E, F\in {\mathcal M}_{t,v,d}$ and $\log m/\log n$ is irrational. The theorem is proved.
\end{proof}


\end{document}